\documentclass{amsart} 
\usepackage{amssymb,amsfonts,amscd,graphicx}
\newtheorem{theorem}{Theorem}[section]
\newtheorem{lemma}[theorem]{Lemma}
\newtheorem{corollary}[theorem]{Corollary}

\theoremstyle{remark}

\theoremstyle{definition}

\numberwithin{equation}{section} \makeatother

\begin{document}


\title[Noncommutative peak interpolation revisited]{Noncommutative peak interpolation revisited}
\author{David P. Blecher}

\address{Department of Mathematics, University of Houston, Houston, TX
77204-3008}
\email[David P.
Blecher]{dblecher@math.uh.edu}

\thanks{Blecher was partially supported by a grant from
the National Science Foundation.}

\maketitle

\begin{abstract}
Peak interpolation is concerned with a foundational kind of
 mathematical task:
 building functions in a fixed algebra $A$ which have prescribed values or behaviour on a fixed closed subset (or on 
several disjoint subsets).  In this paper we do the same but now $A$ is an  algebra of operators on a Hilbert space.
We briefly survey this {\em 
noncommutative peak interpolation}, which we have studied with coauthors
in a long series of papers, and whose basic theory now appears to be 
approaching its culmination.  This program developed 
from, and is based partly on, 
 theorems of Hay and Read whose proofs were spectacular,
but therefore inaccessible to an uncommitted reader.   
We give short proofs of these results, using recent progress in
noncommutative peak interpolation, and conversely give examples of the use of these theorems
in peak interpolation.  For example, we prove a useful new noncommutative peak interpolation
theorem.     
\end{abstract}

\section{Introduction} 

\noindent For us, an operator algebra is a norm closed algebra of operators on a Hilbert space.
In `noncommutative peak interpolation', one generalizes classical `peak interpolation' to the setting of
operator algebras, using Akemann's noncommutative topology \cite{Ake2,Ake,AP}. In classical peak 
interpolation
the setting is a subalgebra $A$ of $C(K)$, the continuous scalar functions on a compact Hausdorff space $K$, and one tries to build
functions  $f \in A$ as in Figure 1 which have prescribed values or behaviour on a fixed closed subset $E \subset K$, and $\Vert f \Vert = \Vert f_{|E} \Vert$.
\begin{figure}[ht]
\centering
\includegraphics[totalheight=1.6in]{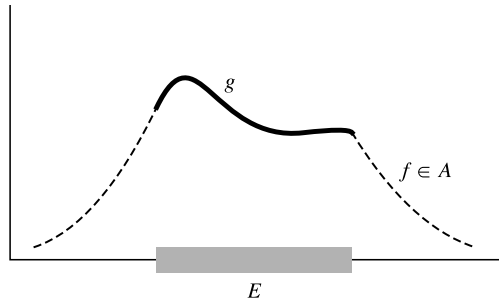}
\caption{Seek $f \in A$ with $f_{\vert E} = g$ or $f  \chi_E = g \, \chi_E$. }
\label{fig1}
\end{figure}
The sets $E$ that `work' for this are the $p$-sets, namely the  closed sets whose
characteristic functions are in $A^{\perp \perp}$.   {\em Glicksberg's
peak set theorem} characterizes these sets as the intersections of {\em peak sets},
i.e.\ sets of form $k^{-1}(\{ 1 \})$ for $k \in A, \Vert k \Vert = 1$. In the separable case they are just the peak sets. 
A typical peak interpolation result,
originating in results of E. Bishop (see e.g.\ e.g.\ II.12.5 in \cite{Gam}), says that if $h 
\in C(K)$ is strictly positive, and if $g$ is a  continuous 
function  on such a set $E$  which is the  restriction of a
function in $A$, and whose absolute value is  dominated 
by the `control function' $h$ on $E$ (see Figure 2A), has an extension $f$ in $A$ satisfying $|f| \leq h$ on 
all of $K$ (see Figure 2B).    
\begin{figure}[ht]
\centering
\includegraphics[totalheight=1.6in]{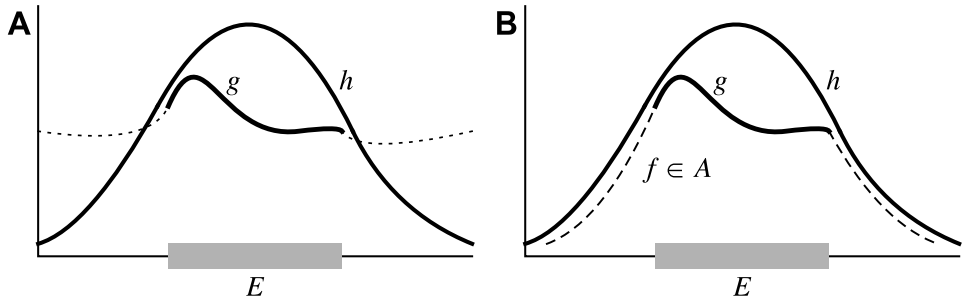}
\caption{{\rm A}: Given `control' $h \geq |g|$ on $E$.  {\rm B}: Seek $f \in A$ with $f_{\vert E} = g$ and $|f| \leq h$ everywhere.}
\label{fig2}
\end{figure}
A special case of interest is
when $h = 1$; for example when this is applied to the 
disk algebra one obtains the well known Rudin-Carleson theorem
(see II.12.6 in \cite{Gam}).    It also yields `Urysohn type lemmas'
in which one finds a function in $A$ which is $1$ on $E$ and close to zero on a 
closed set $F$ disjoint from $E$ (it can be zero on $F$ if $F$ is also a $p$-set).   We discuss below generalizations of all of these results.  

Noncommutative interpolation for $C^*$-algebras  
has been studied by many $C^*$-algebraists, and is  
a key application of Akemann's noncommutative topology.  See particularly L.G. Brown's treatise
\cite{Brown}.  For example 
Akemann's  Urysohn lemma for $C^*$-algebras (see e.g.\ \cite{Ake2}) is a noncommutative
interpolation result of a selfadjoint flavor, and this result plays a role for example in
 recent approaches to the important Cuntz semigroup \cite{ORT}.

  Noncommutative peak interpolation for (possibly nonselfadjoint) operator algebras
was introduced in the thesis of our student Damon Hay \cite{Hayth,Hay}.   
It is the theory one gets if one combines the two theories discussed above: classical peak interpolation and 
the $C^*$-algebra variant.  Here we have 
a subalgebra $A$ of a possibly noncommutative $C^*$-algebra $B$,
and we wish to   build operators in $A$ which have 
prescribed behaviours with respect to Akemann's  noncommutative 
generalizations of closed sets, which are certain projections $q$ in $B^{**}$.   In the 
case that $B = C(K)$, the characteristic function $q = \chi_E$ of an open or  closed set
$E$ in $K$ may be viewed as an element of $C(K)^{**}$ in a natural way since 
$C(K)^*$ is a certain space of measures on $K$.   Via semicontinuity, it is natural 
to declare a projection  $q \in B^{**}$ to be {\em open} if it is a increasing (weak*) limit of 
positive elements in $B$, and closed if its `perp'  $1-q$ is {\em open} (see \cite{Ake2,Ake}).    Thus if $B = C(K)$
the open or closed projections are precisely the   characteristic functions of  open or  closed sets.  Thus one has a 
copy of the topology in the second dual (one may work in  a small subspace of the second dual if one prefers).  We will not discuss   noncommutative topology in detail here, but with the definitions above one can now try to prove inside $B^{**}$ noncommutative versions of the 
basic results in topology, where unions of  open sets are replaced by suprema $\vee_i \, p_i$
of projections, etc.   Open  projections arise naturally in   functional
analysis.  For example, they come naturally out of the `spectral theorem/functional calculus': the  spectral projections of a selfadjoint operator $T$ corresponding to open sets in the spectrum of $T$,  are open projections.  The range projection of  any operator is an open projection.   We state Akemann's noncommutative Urysohn lemma in the case that $B$ is
a unital $C^*$-algebra:  Given $p, q$ closed projections in $B^{**}$, with $p q = 0$, there exists
 $f \in B$ with $0 \leq f \leq 1$ and
$f p =0$  and $f q = q$.   Indeed Akemann's
noncommutative topology has good `separation' properties, whereas other
approaches to noncommutative topology are usually spectacularly far from even being
`Hausdorff'.  Note that the classical statement $f = g$ on $E$ (see Figure 1), 
becomes $f q = gq$ where $q$ is the projection playing the role of 
(the characteristic function of) $E$.  On the other hand, an order relation
like $|f| \leq g$ on $E$, might become the operator theoretic statement $f^* q f \leq g^* q g$, or something similar.
 
Over the years we with coauthors (particularly Hay, Neal, and Read) have developed a number of noncommutative peak interpolation results, which
when specialized to the case $B = C(K)$ collapse to classical peak interpolation theorems.
Moreover, in the course of this investigation 
 striking applications have emerged to the theory of one-sided ideals or
hereditary subalgebras of operator 
algebras, the theory of approximate identities, noncommutative topology, noncommutative function
theory, the
generalization of Hilbert C*-modules to nonselfadjoint algebras,
and other topics (see e.g.\ \cite{BHN}, \cite{BRead},
 \cite{BRII}, \cite{BNI}, \cite{BNII}, \cite{Hayr}, \cite{Ueda}).
    In many of these applications one is mimicking $C^*$-algebra techniques, but using ideas from our theory.  
Current with the most recent version of \cite{BRII} the peak interpolation program appears to be 
approaching its culmination.  The basic theory seems now to be essentially complete, and we 
have a good idea of what works and what does not.  What remains is further applications, and  
in particular those of the kind we have been doing in the cited papers, namely the generalization of more $C^*$-algebraic techniques and 
results to general operator algebras.  Also  there will be applications inspired by
the matching function theory (see e.g.\ \cite{Gam,Stout}).     

Two of the most 
powerful results in the theory are Read's theorem on approximate identities \cite{Read}, and
Hay's main theorem in \cite{Hay,Hayth}.  These are  foundational results in
the subject, but the extreme depth of their proofs hindered their accessibility.  In the present note
we give short proofs of both of these results by using noncommutative peak interpolation.  The proofs are still quite nontrivial, but we have written them
so as to be readable in full detail in an hour or so by a functional analyst.    The crux of the proof is a special case of
a new noncommutative peak interpolation theorem, the latter also proven here,
 generalizing the classical one mentioned above in the first
paragraph of our paper (involving $f$ and $h$).      Below we will give examples of 
applications of both results to peak interpolation.  Indeed our paper is in part a brief 
survey of the basic ideas of noncommutative peak interpolation.  

Turning to notation, we will delay  many of the noncommutative definitions and features until they are needed.  The reader is referred for example to \cite{BLM,BHN,BRead} for more details on
some of the topics below if needed.  We will use silently the fact from basic analysis that
$X^{\perp \perp}$ is the weak* closure in $Y^{**}$ of
a subspace  $X \subset
Y$, and is isometric to $X^{**}$.   Recall that 
$X$ is an {\em $M$-ideal} in $Y$ if $X^{\perp \perp} \oplus_\infty L = Y^{**}$
for a subspace $L$ of $ Y^{**}$.  This notion was introduced by Alfsen and Effros;
and in this case for any $y \in Y$ there exists 
$x \in X$ with the distance $d(y,X) = \Vert y -x \Vert$ (see 
\cite{HWW}).     For us a {\em projection}
is always an orthogonal projection.  An {\em  approximately unital} operator algebra
is one that has a contractive approximate identity (cai).
 If $A$ is a nonunital
operator algebra represented (completely) isometrically on a Hilbert
space $H$ then one may identify the unitization 
$A^1$ with $A + {\mathbb C} I_H$.   If $A$ is unital (i.e.\ has an identity
of norm $1$) we set $A^1 = A$.  If $A$ is an 
operator algebra then the  
second dual $A^{**}$ is  an operator algebra too with its (unique)
Arens product, this is also the product inherited from the von Neumann
algebra $B^{**}$ if $A$ is a subalgebra of a $C^*$-algebra $B$. 

\section{Read's theorem}

In the following, $A$ is an operator algebra with cai. Let $C$ be any $C^*$-algebra generated by $A$, which
has the same cai by \cite[Lemma 2.1.7 (2)]{BLM}, and let $B = C^1$,
which is a $C^*$-algebra generated by $A^1$. Let $e$ be the weak* limit 
in $(A^1)^{**}$ of the cai, so $e = 1_{C^{**}}$, and let
$q = 1-e$.  Both projections are in the center of $B^{**}$, since $e (\lambda 1 + c) = \lambda e + c = (\lambda 1 + c)e$ for $\lambda \in {\mathbb C}, c \in C$.
  
We first prove a simple noncommutative peak interpolation result that has implications for the unitization of an operator algebra.

\begin{lemma} \label{peakthang222} Suppose that $A$ is an approximately unital operator algebra, and let
$q,e, C, B$ be as above.  
If $q \leq d$ for an invertible $d$ in the positive cone $B_+$,
then there exists an element $g \in A^1$ with $g q = q g = q$, and $g^* g \leq d$.
Thus if   $A$ is nonunital  and $c \in C_+$ with $\Vert c \Vert < 1$ then there
exists an $a \in A$ with $|1 + a|^2 \leq 1-c$.
\end{lemma}
\begin{proof} Let  $f = d^{-\frac{1}{2}}$.   
Since the `second perp' is the weak* closure,
we have $(A^1f)^{**} = (A^1 f)^{\perp \perp} = (A^1)^{\perp \perp} f$.
Multiplication by the central projection $e = 1-q$ (resp.\ by $q$) is a contractive projection on $(A^1)^{\perp \perp} f$
whose range is $e (A^1)^{\perp \perp} f = A^{\perp \perp} f =
(A f)^{\perp \perp}$ (resp.\ is $q A^{\perp \perp} f$), which may be viewed as 
a subspace of $eB^{**} \oplus_\infty q B^{**}$. So $A f$ is  
an  $M$-ideal in $A^1 f$ as defined in the introduction, 
and, as we said there, there exists $y \in A$ such that
$\Vert f - yf \Vert = d(f,Af)$.  Since 
$$A^1f/Af \subset (A^1f)^{**} / (A f)^{\perp \perp} =
(A^1)^{\perp \perp} f / (e(A^1)^{\perp \perp} f) \cong 
q (A^1)^{\perp \perp} f ,$$
we have $d(f,Af) = \Vert q f \Vert = \Vert f q f \Vert^{\frac{1}{2}} \leq 1.$
Setting $g = 1-y$ then $q g = g q = q$, 
and $\Vert g f \Vert \leq 1,$
so that $f g^* g f \leq 1$ and so $g^* g \leq d$.
For the last assertion take $d = 1-c$.
\end{proof}

{\sc Remark.}  The last assertion of Lemma \ref{peakthang222}  is in fact equivalent to the other assertion.  Since this will not be needed we leave it as an exercise in spectral theory.

\begin{theorem}[Read's theorem on approximate identities]
 \label{readsps} If $A$ is an 
operator algebra with a cai, then $A$ has a cai $(e_t)$ with positive real parts, satisfying
$\Vert 1 - 2 e_t \Vert \leq 1$ for all $t$.
\end{theorem}

\begin{proof} 
Let $q,e, C, B$ be as above.
Then $e$ is an open projection in the $C^*$-algebra sense
with respect to $B$. 
Indeed, any increasing cai $(b_t)$ for $C$ is a net of positive elements in $B$
increasing to $e$.  Note that $q(1-b_t) =
(1-b_t) q = q$. Consider
the net $(f_s) = (\frac{1}{n} 1 + \frac{n-1}{n} (1-b_t))$, which is a net of strictly positive elements $f_s$
in Ball$(B)$ with weak* limit
$q$.  Here $s = (t,n)$.  We have $q f_s = f_s q = q$, and $f_s \geq q$. By Lemma \ref{peakthang222},
there exists $a_s \in A^1$ with
$a_s q = q a_s = q$ and $a_s^* a_s \leq f_s \leq 1$. We have
$(a_s - q)^* (a_s - q) = a_s^* a_s - q \leq f_s - q \to 0$ weak*.
Using the universal representation we may view $B^{**}$ as a von Neumann
algebra in $B(H)$ in such a way that the weak* topology of $B^{**}$ coincides with the $\sigma$-weak
topology.
Then $f_s \to q$ WOT, and so for $\zeta, \eta \in {\rm Ball}(H)$ we have
$$|\langle (a_s - q) \zeta, \eta \rangle|^2 \leq \Vert (a_s - q) \zeta \Vert^2 =
\langle (a_s - q)^* (a_s - q) \zeta, \zeta \rangle \leq \langle (f_s - q) \zeta, \zeta \rangle \to 0 .$$
So $a_s \to q$ WOT, and hence $a_s \to q$ weak* since these are bounded. If $u_s = 1 - a_s$ then 
$e u_s = u_s e = u_s,$ so $u_s \in A$, indeed $u_s$ is in the convex subset ${\mathfrak F}_A = \{ a \in A : 
\Vert 1 - a  \Vert \leq 1 \}$ of $2 {\rm Ball}(A)$.
Also $u_s \to e$ weak*, so $(x u_s)$ and $(u_s x)$ converge weakly to $x$ for any $x \in A$.   
Applying a standard convexity argument: for $x_1, \cdots, x_m \in A$,  
the norm and weak closures of the convex set 
$$F = \{(x_1 u - x_1, \cdots ,  x_m u - x_m, u x_1 - x_1, \cdots, u x_m  - x_m) : u \in {\mathfrak F}_A \}$$
coincide by Mazur's theorem, and contain $0$, from which it follows 
that $A$ has a bounded approximate identity $(e_r)$ in ${\mathfrak F}_A$ (see e.g.\
the last part of the first paragraph
of the proof of \cite[Theorem 6.1]{BHN} for more details if needed). 

A result of the author and Read states that $\frac{1}{2} {\mathfrak F}_A$ is closed under $n$th roots \cite[Proposition 2.3]{BRead}, so
$((\frac{1}{2} e_r)^{\frac{1}{n}}) \subset \frac{1}{2} {\mathfrak F}_A$. Since $(\frac{1}{2}
e_r)^{\frac{1}{n}} e_r \to e_r$ with $n$ (see below), it is easy to see that $((\frac{1}{2}
e_r)^{\frac{1}{n}})$ is a cai
for $A$ from $\frac{1}{2} {\mathfrak F}_A$. For example, if $a \in A$,
then $$\Vert a - (\frac{1}{2} e_r)^{\frac{1}{n}} a \Vert \leq \Vert a - e_r a \Vert + \Vert (e_r -
(\frac{1}{2} e_r)^{\frac{1}{n}} e_r) a \Vert + \Vert (\frac{1}{2} e_r)^{\frac{1}{n}} (a - e_r a) \Vert \to
0$$ with $n$ and $r$,
since the middle term here is bounded by $2 \Vert a \Vert$ times
$$\frac{1}{2} \Vert e_r - (\frac{1}{2} e_r)^{\frac{1}{n}} e_r \Vert \leq \Vert p_n \Vert_{\bar{{\mathbb D}}} \to 0$$
with $n$, using von Neumann's inequality by applying to $1 - e_t$ the function $p_n(z) = 1 - z - (1-z)^{1 + \frac{1}{n}}$ on the
unit disk in ${\mathbb C}$.
\end{proof}

{\sc Remark.}  The elements $e_t$ in the cai above are as close as one wishes 
to being positive.  Indeed as explained in \cite[Theorem 2.4]{BRead}, 
the $n$th roots in the proof 
have numerical range in a horizontal `cigar' centered on $[0,1]$ 
in the plane, which is as thin as one wishes.  

\medskip

A sample application of Read's theorem to noncommutative peak interpolation:
the point in the last 
proof where we show that $A$ has a bounded approximate identity in ${\mathfrak F}_A$
already solves the main open question that arose in Hay's thesis \cite{Hayth,Hay}: the validity of the noncommutative version of Glicksberg's peak set theorem
mentioned in the first paragraph of the paper, by \cite[Theorem 6.1]{BHN} and the surrounding discussion.
That is, the closed projections in $B^{**}$ which lie in $A^{\perp \perp}$, are precisely the `infs' of 
{\em peak projections} (and in the separable case they are just the peak projections).  The latter are Hay's noncommutative generalization
of peak sets and have many characterizations in the papers cited below.  The following evocative characterization is 
proved at the end of the introduction of \cite{BRII}:
For any operator algebra $A$, the peak projections are the weak* limits of $a^n$ for $a \in {\rm Ball}(A)$ in the cases that such 
limit exists.

\section{Peak interpolation and Hay's theorem}

The following general functional analytic
lemma is due to the author and Hay \cite[Proposition 3.1]{Hay} (its proof  follows the lines of
\cite[Lemma II.12.3]{Gam}).

\begin{lemma} \label{dist}  Let $X$ be a closed subspace of a unital $C^*$-algebra $B$, and let $q \in B^{**}$ be a closed
projection such that $(qx)(\varphi) = 0$ whenever $\varphi \in X^\perp, x \in X$.  Let
$I = \{ x \in X : qx = 0 \}$.  Then the distance $d(x,I) = \Vert qx \Vert$ for all $x \in X$.
\end{lemma}

The following result (\cite[Proposition 3.2]{Hay}) follows easily from Lemma \ref{dist}
following the lines of
\cite[Lemma II.12.4]{Gam}.  It is an  `approximate interpolation' result, indeed  it is an `error  epsilon' variant of 
the Bishop type interpolation result seen in and above Figure 2 above.

\begin{corollary} \label{Haything}   If $X$ and $q$ satisfy the conditions in the last result, 
and keeping the notation there, if $d$ is a positive invertible element of $B$ with $d \geq a^* q a$ for some $a \in X$,
and if $\epsilon > 0$, then there exists $b \in X$ with $qb = qa$ and $b^* b \leq d + \epsilon 1$.
\end{corollary}

\begin{theorem}[Hay's theorem on one-sided ideals]  \label{haysps}  If $A$ is a unital 
subalgebra of a unital $C^*$-algebra $B$, then the right ideals $J$ in $A$ which have a left
contractive approximate identity, are precisely the right ideals $\{ a \in A : a = pa \}$ for an 
open projection $p \in B^{**}$ which lies in $A^{\perp \perp}$.
If these hold then $J^{\perp \perp} = p A^{**}$.
\end{theorem}

\begin{proof}   As explained in \cite{Hay}, this is easy and standard
functional analysis, most of it working 
in any Arens regular Banach algebra, except for the following Claim: 
if an open projection $p \in B^{**}$  lies in $A^{\perp \perp}$, then 
$p$ is a weak* limit of a net $(x_s) \subset A$ satisfying 
$p x_s = x_s$.  For example, if the Claim holds then $p$ is in the weak* closure
of the right ideal $I = \{ a \in A : a = pa \}$, and is a left identity
for that weak* closure, and then it is well known (see e.g.\ \cite[Proposition 2.5.8]{BLM}) that
$I$ has a left cai.    

To prove the Claim,
note that $q = 1-p$ is closed.  As in the first lines of the proof of Theorem \ref{readsps},
 there is a net in $B$ of strictly positive $f_t \searrow q$.  Let $f^s = f_t + \frac{1}{n} 1$,
where $s = (t,n)$.  By Corollary \ref{Haything} with $X = A$ and $a = 1$, there exists $a_s \in A$ with
$q a_s = q$ and $a_s^* a_s \leq f^s$. 
Then $a_s \to q$ weak* as in the proof of Theorem \ref{readsps},
so $x_s = 1-a_s$ satisfies $p x_s = x_s \to p$ weak*.    \end{proof}

 We will use the last result, Hay's theorem, to prove a rather general peak interpolation result, a noncommutative generalization of the Bishop type interpolation result seen in and above Figure 2 above.  It contains as  a special case Lemma \ref{peakthang222}, which we proved in the previous Section.  The reader may want to use that proof as a guide since it contains some of the ideas and strategy in a simpler setting.  

\begin{theorem} \label{peakthang22}   Suppose that $A$ is an operator algebra
(not necessarily approximately unital), a subalgebra of a unital $C^*$-algebra $B$.
Identify $A^1 = A + {\mathbb C} 1_B$.  Suppose
 that  $q$ is a closed projection in $B^{**}$ which lies in $(A^1)^{\perp \perp}$.
   If $b \in A$ with $b q= qb$,
 and $q b^* b q  \leq q d$ for an invertible positive $d \in B$ which commutes with $q$,
then  there exists an element $g \in A$ with $g q =  q g = b q$, and $g^* g \leq d$.
\end{theorem}

\begin{proof}     
Let $\tilde{D} = (1-q) (A^1)^{**} (1-q) \cap A^1$,
let $C$ be the closed subalgebra of $A^1$ generated by $\tilde{D}, b,$ and $1$, and let 
$$D = \{ x \in C \cap A : q x = 0 \} = \tilde{D} \cap A \cap C \subset A .$$   By Hay's theorem above
$1-q$  is a limit
of $x_s = (1-q) x_s \in A^1$.  Indeed in the language of the last proof,
$1-q \in J^{\perp \perp}  = p A^{**}$, so there exists $x_s \in J = 
\{ a \in A : a = (1-q) a \}$ with $x_s \to 1-q$ weak*.
 By symmetry  $1-q$  is a limit
of $z_t = z_t (1-q)  \in A^1$.  Hence $1-q$  is a limit of $x_s z_t \in \tilde{D}$, that is $1-q \in 
\tilde{D}^{\perp \perp} \subset C^{\perp \perp}$.
Thus  $q \in C^{\perp \perp}$.   Clearly 
$\tilde{D}^{\perp \perp}  \subset (1-q) C^{\perp \perp}.$  Conversely, since 
$\tilde{D}$ is an ideal in $C$, so that $\tilde{D}^{\perp \perp}$ is an ideal in $C^{\perp \perp}$, we have
$(1-q) C^{\perp \perp}  \subset \tilde{D}^{\perp \perp}$.  So 
$$(\tilde{D}f)^{\perp \perp} = \tilde{D}^{\perp \perp} f = (1-q) C^{\perp \perp} f = (1-q) (C f)^{\perp \perp},$$
 and so $\tilde{D}f$ is an $M$-ideal in $Cf$, using the fact that 
$q$ is a central projection in $C^{\perp \perp}$.
 The associated
$L$-projection $P$ onto the subspace $(\tilde{D} f)^\perp$ of $(Cf)^*$, is multiplication by $q$.
Let  $x \in (C \cap A) f$ and
$\varphi \in ((C \cap A) f)^\perp$, and let $(c_t)$ be a net in
$C$ with weak* limit $q$. Then $q \varphi(x) = \lim_t \varphi(c_t x) = 0,$
since $c_t x \in (C \cap A) f$ (because $C (C \cap A) \subset (C \cap A)$).   
We will make two deductions from this.
First, $P(((C \cap A)f)^\perp) \subset ((C \cap A)f)^\perp$.
So by \cite[Proposition I.1.16]{HWW}, we have that $Df$ is an $M$-ideal in $(C \cap A) f$.
Second, we deduce from  Lemma \ref{dist} with $X = (C \cap A) f$, that
$d(x,D f) = \Vert q x \Vert$ for all $x \in (C \cap A) f$.
So $d(bf,Df) = \Vert q bf \Vert   = \Vert f q b^* b q f \Vert^{\frac{1}{2}} \leq 1$.
By the distance formula in the introduction
there exists $y f \in D f$ such that 
$\Vert bf - yf \Vert   \leq 1.$
   Setting $g = b-y$ then $q g = g q = q b$, and $\Vert g f \Vert \leq 1,$ so that $f g^* g f \leq 1$ and  $g^* g \leq d$. 
\end{proof} 

Let us see that Theorem \ref{peakthang22}
 implies the classical Bishop-type peak interpolation result in the first 
paragraph of the paper (Figure 2), and therefore a `commutativising'  of the proof above 
gives a new and quick proof of that classical result.  If $B = C(K)$, and if 
$E$ is a peak or $p$-set in $K$ for $A \subset C(K)$, 
then by what we said in in the first three paragraphs
 of the introduction, the characteristic function of $E$ may be 
viewed as a closed projection $q$ in $A^{**}
=  A^{\perp \perp} \subset C(K)^{**}$.
The condition that $q b^* b q  \leq q d$ is
saying precisely that the strictly positive function $d \in C(K)$ dominates 
$|b|^2$ on $E$.  Thus Theorem \ref{peakthang22} gives $g \in A$ with $|g|^2 \leq d$
on $K$, and $g = b$ on $E$ (since $g \chi_E = b \chi_E$).   

In \cite{BRII} it is shown that one cannot drop the condition $bq = qb$ in the
last result.  It is easy to see one also cannot drop the condition $dq = qd$
(a counterexample: $A = {\mathbb C} q, b = q$ for a nontrivial projection $q \in M_2$).

The noncommutative peak interpolation theorem \ref{peakthang22}
should have many  applications. For example in the last section of \cite{BRII}
a special case of 
it is used to develop the theory of compact projections in algebras not necessarily 
having  any kind of approximate identity.  
 Using this special case we obtain a generalization
of Glicksberg's peak set theorem
mentioned in the first paragraph of the paper.
That is, even if $A$ has no approximate identity,
 the compact projections relative to $A$, that is the closed projections with
respect to a unitization, which lie in $A^{\perp \perp}$,
are precisely the `infs' of
the peak projections discussed at the end of the last section.  If $A$ is
separable then the compact projections relative to $A$ are just the peak projections
(see e.g.\ \cite[Lemma 1.3 and Proposition 6.4 (2)]{BRII}).   
We also obtain noncommutative Urysohn type lemmas in that setting (that is, 
given compact $q$ relative to $A$,
with $q \leq p$ open, there exists $a \in {\rm Ball}(A)$
 with $aq = qa = q$ and $a$ `small' on $1-p$.).    
   Indeed we show that 
these results follow from the case $d = 1$ of 
Theorem \ref{peakthang22}.   We also use  ideas from
papers with Neal and Read \cite{BNII,BRead}, which in turn use Read's theorem.
In the last mentioned  noncommutative Urysohn type lemma one may have $a$ satisfy
$||1-2a|| \leq 1$, and $a$ `equal to zero' on $1-p$,
that is $a(1-p) = (1-p)a = 0$, if $p \in A^{\perp \perp}$ (this follows from
Theorem 2.6 in \cite{BNII}).

Finally, we remark that
simple Tietze theorems of the flavour of the Rudin-Carleson theorem mentioned in the
 first paragraphs of this paper, follow from our interpolation theorems by adding a
hypothesis of the kind in Proposition 3.4 of \cite{Hay}. 

 \smallskip

{\sc Acknowledgements.}  We thank Damon Hay for correcting several typos in the first draft, and thank Manuel Lopez for coding the figures.

\end{document}